\documentclass[12pt]{amsart}
\usepackage{amsmath,amssymb}
\usepackage{amsfonts}
\usepackage{amsthm}
\usepackage{latexsym}
\usepackage{graphicx}
\usepackage{epstopdf}
\usepackage{caption}
\usepackage{subcaption}
\usepackage{cite}
\usepackage{contour,color}

\def\p{\partial}

\def\R{\mathbb{R}}

\def\vv<#1>{\langle#1\rangle}

\def\s1{{\mathbb{S}^1}}

\def\XXint#1#2{\setbox0=\hbox{$#1{#2}{\int}$}{#2}\kern-.5\wd0 }

\def\XXint#1#2#3{{\setbox0=\hbox{$#1{#2#3}{\int}$}
     \vcenter{\hbox{$#2#3$}}\kern-.5\wd0}}

\def\vv<#1>{{\left\langle#1\right\rangle}}

\newtheorem{thm}{Theorem}[section]

\newtheorem{cor}{Corollary}[section]
\theoremstyle{definition}

\theoremstyle{remark}

\numberwithin{equation}{section}

\begin{document}
\title{Quasilocal energy and surface geometry of Kerr spacetime}
\author{Chengjie Yu}
\address{Department of Mathematics, Shantou University, Shantou, Guangdong, 515063, China}
\email{cjyu@stu.edu.cn}
\author{Jian-Liang Liu}
\address{Department of Mathematics, Shantou University, Shantou, Guangdong, 515063, China}
\email{liujl@stu.edu.cn}
\renewcommand{\subjclassname}{%
  \textup{2010} Mathematics Subject Classification}
\date{}
\begin{abstract}
We study the quasi-local energy (QLE) and the surface geometry for Kerr spacetime in the Boyer-Lindquist coordinates without taking the slow rotation approximation. We also consider in the region $r\leq2m$, which is inside the ergosphere. For a certain region, $r>r_{k}(a)$, the Gaussian curvature of the surface with constant $t,r$ is positive, and for $r>\sqrt{3}a$ the critical value of the QLE is positive. We found that the three curves: the outer horizon $r=r_{+}(a)$, $r=r_{k}(a)$ and $r=\sqrt{3}a$ intersect at the point $a=\sqrt{3}m/2$, which is the limit for the horizon to be isometrically embedded into $\mathbb{R}^3$
. The numerical result indicates that the Kerr QLE is monotonically decreasing to the ADM $m$ from the region inside the ergosphere to large $r$. Based on the second law of black hole dynamics, the QLE is increasing with respect to the irreducible mass $M_{\mathrm{ir}}$. From a results of Chen-Wang-Yau
, we conclude that in a certain region, $r>r_{h}(a)$, the critical value of the Kerr QLE is a global minimum.
\end{abstract}
\maketitle\markboth{Yu \& Liu}{Quasi-local energy of the Kerr spacetime}

\section{Introduction}

In general relativity, due to the equivalence principle, gravitational energy has no proper local description. Although there is a well known total energy defined by Arnowitt, Deser and Misner, the ADM mass, we need some definitions for the practical applications which involve a finite region rather than the whole space. This leads to the concept of quasi-local energy. There are many different proposals for defining quasi-local quantities. Here we follow the covariant Hamiltonian approach of Chen, Nester and Tung (CNT) \cite{CNT2,CNT3,CNT4}, and use the method in \cite{SCLN2} to determine the reference. The physical significance of the choice of reference is the choice of the \emph{zero-point} value of energy [\cite{BLY} Sec. V, A. Subtraction term]. 
 For a survey of the covariant Hamiltonian approach one may refer to \cite{CNT5}. A major difficulty for the Hamiltonian approach is to find a suitable way to identify the reference. Basically there are two ways to determine the reference: The first one is ``analytic matching,'' in which the reference variables are directly restricted from the physical variables, for example, in Schwarzschild spacetime, one can calculate the metric and connection for the physical variables and then take the $m=0$ limit to be the reference ones \cite{LCN}. Although it is sometimes convenient to do this calculation, it is not so clear as to what is the meaning of the reference one gets; the second method is ``4D isometric matching,'' which is based on the 2-surface $S$ isometric embedding into a reference spacetime\footnote{We assume that the surface $S$ and the embedded surface $\bar{S}$ have the same orientation, and also that the spacetime $M$ and the reference spacetime $\bar{M}$ have the same time orientation and space orientation.} and then making an extension in the ``tubular'' neighborhood of $S$ in the normal plane. One can identify the local 4-frame of the physical and the reference spacetime \emph{just on $S$}. For an asymptotic flat spacetime,\footnote{The 2-surface isometric embedding into Minkowski spacetime is based on Wang and Yau \cite{WaYa}.} a resonable choice of reference is Minkowski spacetime, and for measuring the quasi-local \emph{energy}, the corresponding displacement vector field on $S$ is the timelike vector field $N$ identical to the timelike Killing vector field $\partial_{T}$ on $\bar{S}$ of the reference spacetime through the 4D isometric matching \cite {LY}. In \cite{LY} we found that for a specific decomposition of $N$, the result is the same as the Wang-Yau energy \cite{WaYa}. One may refer to the surveys on the Wang-Yau energy \cite{Miao,W,Sz}.

The application to Kerr spacetime using this method was considered in \cite{SCLN2}. The results are (i) for $r\geq2m$, the critical value of the QLE agrees with Martinez's result \cite{Ma} for the slow rotation approximation. (ii) The numerical results implies that the QLE is decreasing for increasing $a$, and monotonically decreasing to the value $m$ with respect to the spatial radius $r$. (iii) The value of the quasi-local angular momentum is the \emph{constant} $am$. 
Later, Tam and one of the authors found that this critical value of the QLE is actually the Brown-York mass [\cite{LT}, Theorem 2.1] (for $a\leq m$), and also the critical point is the unique solution in the region $r_{+}<r<8m/3$ for the slow rotation approximation (the global uniqueness is still not clear). Here $r_{+}=m+\sqrt{m^2-a^2}$ is the outer horizon of the Kerr spacetime.

In this paper, we analyze the QLE of Kerr spacetime for unrestricted rotation and also the region inside the ergosphere.\footnote{The ergosphere is the \emph{static limit}, given by $r=m+\sqrt{m^2-a^2\cos^2\theta}$ [\cite{MTW}, Box 33.2].} The situation is different from the case $r\geq 2m$ and slow rotation. For $r$ less than $2m$, there are limits on the isometric embedding; for nonslow rotation, there is no known theorem for positivity in some region. 

As mentioned before, the critical value of the QLE we considered depends on the 2-surface isometric embedding into $\mathbb{R}^3$. There are limits on $a$ for the surface's isometric embedding. It is known that in the Kerr spacetime, there is no event horizon for $a>m$, and hence the singularity is naked. For $a\leq m$ but not small, Smarr found that if $a>\sqrt{3}m/2$, the Gauss curvature is negative near the poles, which implies that the outer horizon cannot be embeded into $\mathbb{R}^3$ \emph{isometrically} \cite{Sm}. One can imagine that when $a$ grows up over that limit, the horizon ``warps'' out of $\mathbb{R}^3$ from the poles [\cite{Sm}, Fig.\ 4] until the extreme limit ($a=m$), and there is no horizon for $a>m$. Something like tearing a sticker  off a table: part of it is off but still some part stays on the table. 

We are interested in the limit not just for the horizon. The main results of this paper are: (i) the Gauss curvature is positive if and only if $r>r_{k}(a)$ [see \eqref{eqn-rk}], where $r_{k}(a)$ is the unique real root of $r^3+a^2 r-6a^2m$. This also gives the limit for the isometric embedding; (ii) the \emph{integrand} of the critical value of the QLE, i.e.\ $k_{0}-k$, is positive if and only if $r>\sqrt{3}a$. Based on these results, we conclude that in a certain region $r>r_{h}(a)$ [see \eqref{eqn-rh}], the critical value of the Kerr QLE is a global minimum with respect to the isometric embedding, which follows as a corollary from Chen-Wang-Yau's result \cite{CWY}; (iii) the numerical results imply that the QLE is increasing with respect to the irreducible mass of the Kerr black hole. 

Note that the three curves: the outer horizon $r_{+}(a)$, $r_{k}(a)$ and $r=\sqrt{3}a$ intersect at the point $a=\sqrt{3}m/2$. In the \emph{triangle-like} region (see Fig.\ 1), $r_{k}(a)\leq r\leq\sqrt{3}a$ for $\sqrt{3}m/2\leq a\leq m$, the integrand $k_{0}-k$ is not always positive. That means we cannot use the theorem in \cite{CWY} to prove the minimizing property in that region. Also it is not mathematically proved that the QLE is positive in that region. However, the numerical results imply positivity and a monotonically decreasing property. This extends the results of \cite{SCLN2} into the region $r\leq 2m$, and also generalizes the result of Martinez \cite{Ma} to nonslow rotation.

Regarding the result that the QLE is decreasing with respect to $a$, we found that if we replace $a$ by the irreducible mass $M_{\mathrm{ir}}$, the QLE is monotonically increasing with respect to $M_{\mathrm{ir}}$ (see Fig.\ 2). The irreducible mass is protortional to the square root of the black hole area, that implies that the larger the black hole, the larger the QLE.

\begin{figure}[!htb]
	\centering
	\begin{subfigure}{0.5\textwidth}
		\centering
		\includegraphics[width=0.92\linewidth, height=0.27\textheight]{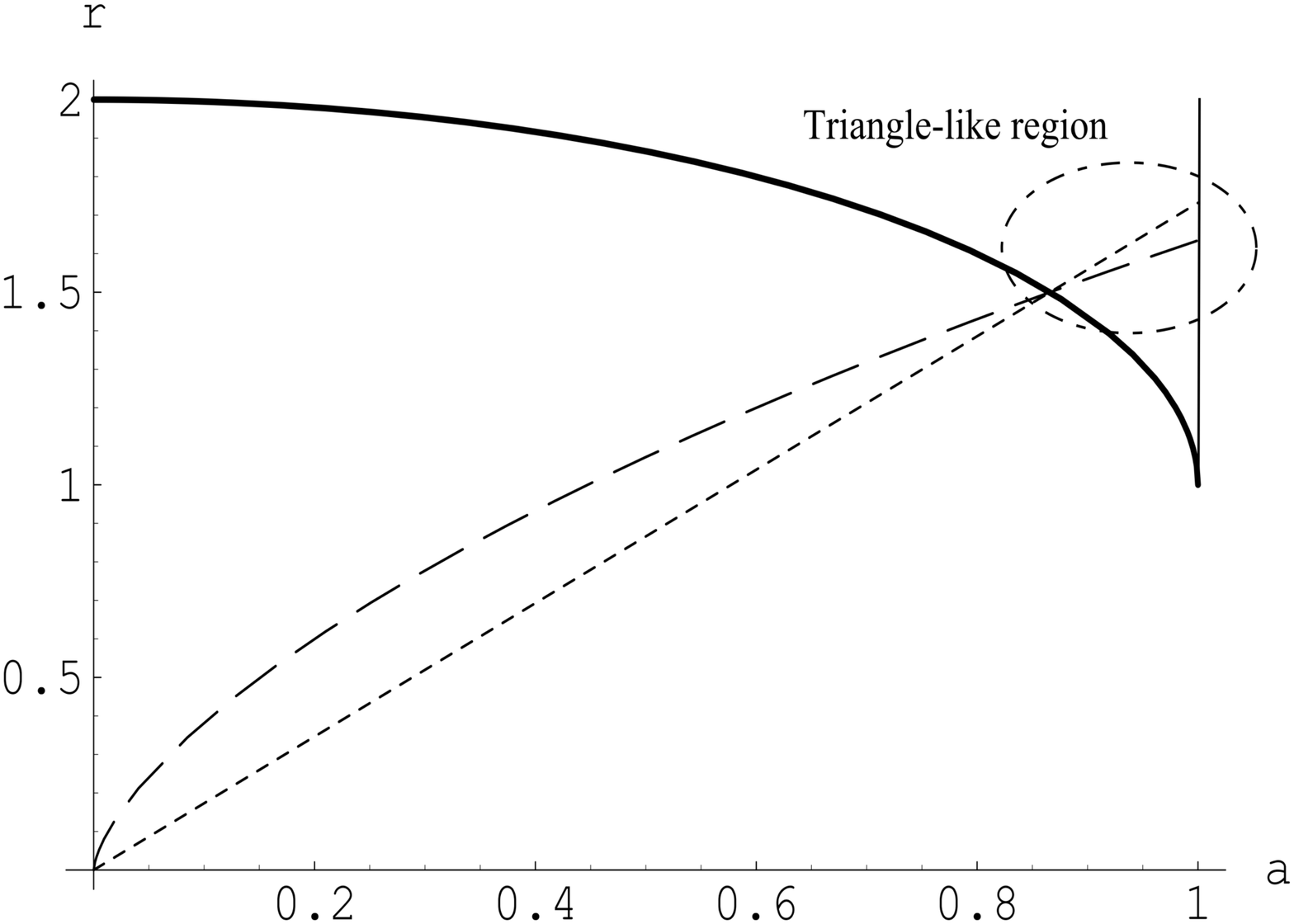}
		\caption{{\scriptsize$r_k,r=\sqrt{3}a$ and $r_{+}$}}\label{fig:4a}		 
	\end{subfigure}%
		\begin{subfigure}{0.5\textwidth}
		\centering
		\includegraphics[width=0.9\linewidth, height=0.27\textheight]{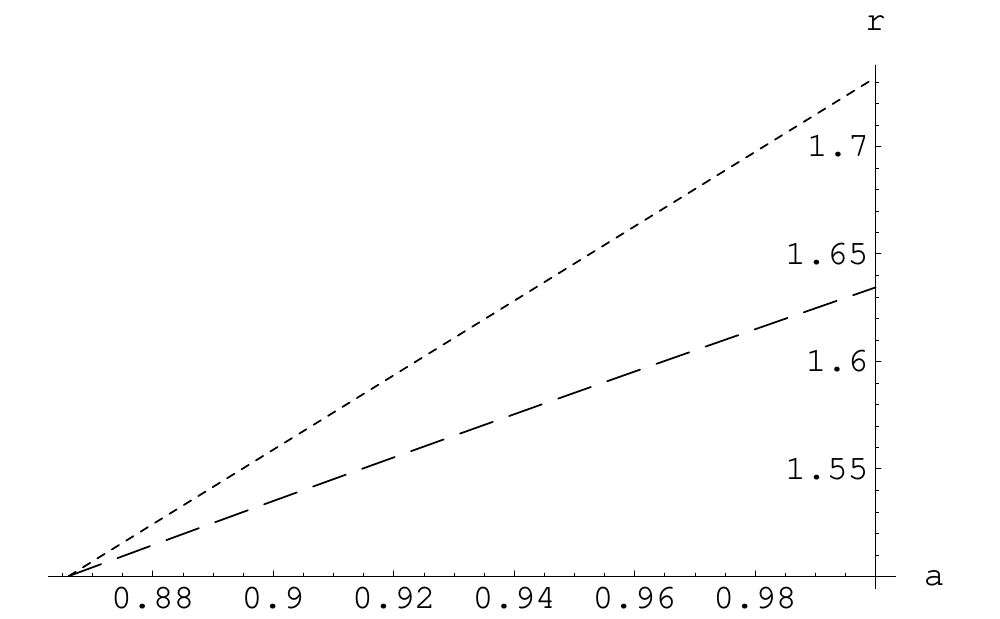}
		\caption{{\scriptsize Triangle-like region $\mathfrak{A}$}}\label{fig:4b}
	\end{subfigure}
	\caption{{\it r--a} figure at $\theta=0$: $r_k$: dashed, $r=\sqrt{3}a$: dotted, $r_{+}$: thickness. }\label{fig:1}
\end{figure}

\begin{figure}[!htb]
	\centering
	\begin{subfigure}{0.5\textwidth}
		\centering
		\includegraphics[width=0.95\linewidth, height=0.27\textheight]{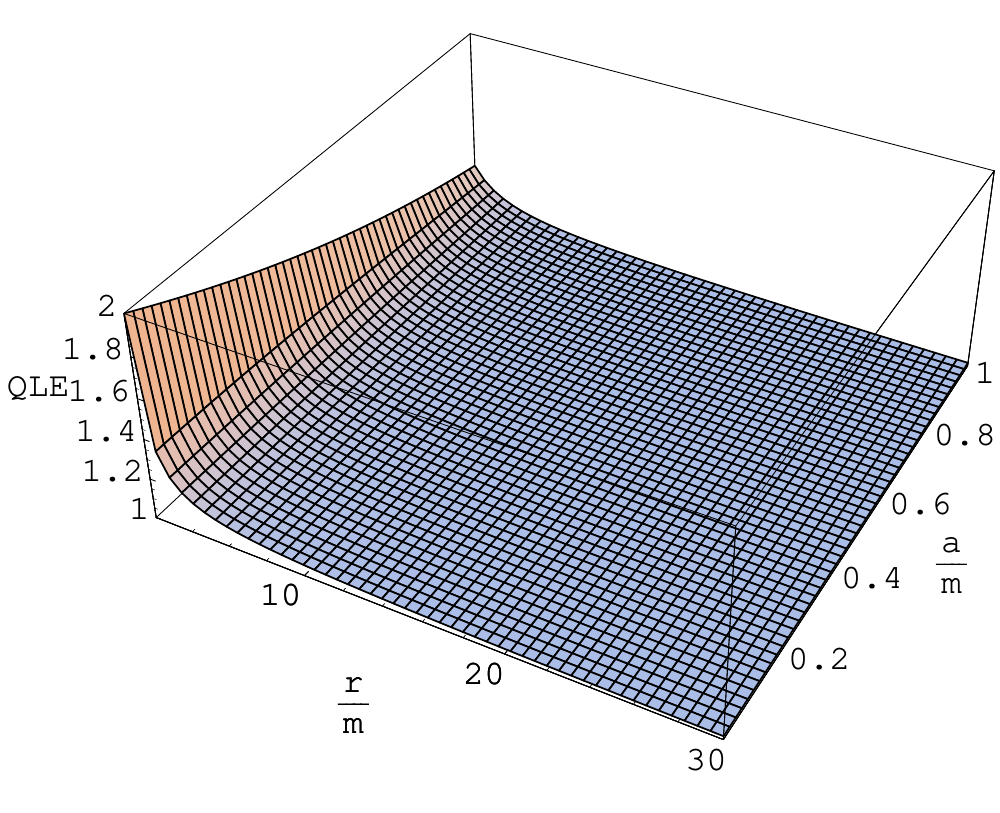}
		\caption{{\scriptsize$0\leq a\leq 1$}}\label{fig:3a}		 
	\end{subfigure}%
		\begin{subfigure}{0.5\textwidth}
		\centering
		\includegraphics[width=0.95\linewidth, height=0.27\textheight]{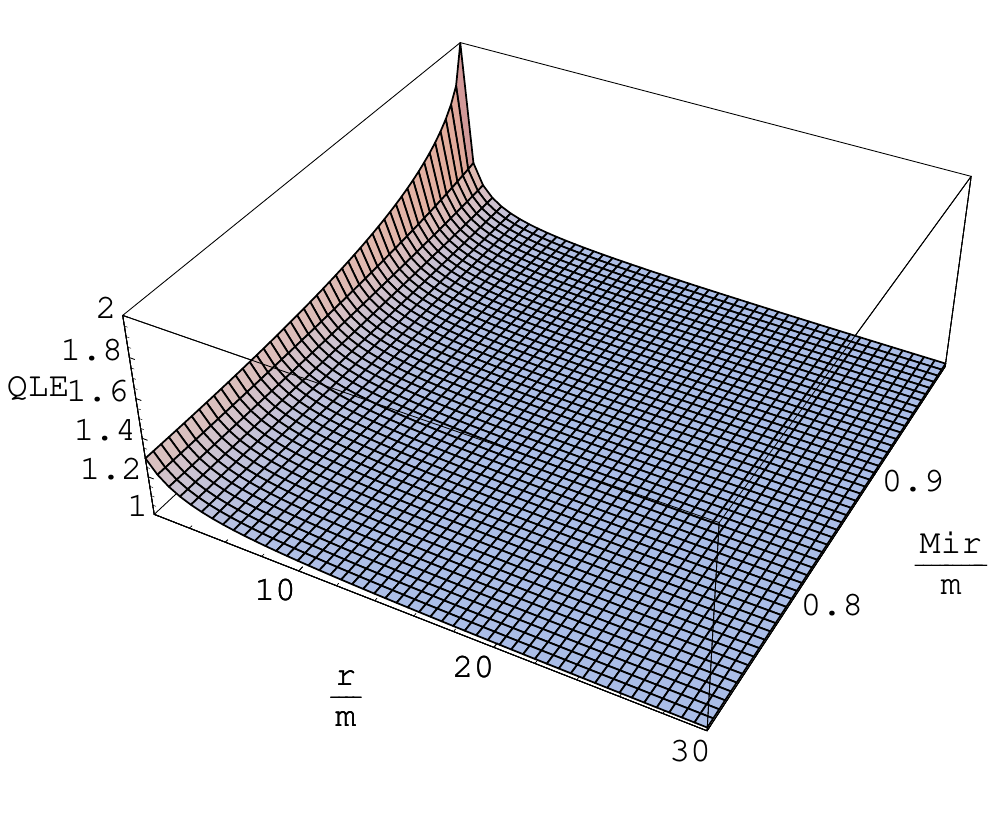}
		\caption{{\scriptsize $\sqrt{2}/2\leq M_{\mathrm{ir}}\leq 1$}}\label{fig:3b}
	\end{subfigure}
	\caption{\hspace{75pt}$r\ge2$, $m=1$.\\These figures show the numerical approximation of the QLE outside $r=2m$. (A) agrees with the result in \cite{SCLN2}; (B) is the result for the irreducible mass $M_{\mathrm{ir}}$.}\label{fig:2}
\end{figure}

\section{Geometrical construction of quasi-local energy}
Let $(M^4,g)$ and $(\bar{M}^4,\bar{g})$ be two oriented and time-oriented spacetimes which are considered as the physical spacetime and the reference spacetime, respectively. Let $S^2$ be a closed spacelike surface in $M$. A 4D isometric matching reference of $S$ is defined as a smooth embedding $\varphi:U\to \bar{M}$ preserving orientations and time-orientations on some open neighborhood $U$ of $S$ such that
\begin{equation}
\varphi^*\bar{g}=g
\end{equation}
on $S$. Let $N$ be a future directed timelike vector field on $M$. Then, the CNT quasi-local energy (see \cite{CNT2,CNT3,CNT4,NCLS,SCLN2,SCLN3,Wu}) of $S$ with respect to $N$ with a 4D isometric matching reference $\varphi$ is given by
\begin{equation}\label{eqn-CNT-E-1}
E(S,N,\varphi)=\frac{1}{8\pi}\int_S \iota^*[(\omega^a{}_b-\bar{\omega}^a{}_b)\wedge i_N\eta_{a}{}^b].
\end{equation}
Here, $\omega^a{}_b$ and $\bar{\omega}^a{}_b$ are the connection forms of the Levi-Civita connections for $g$ and $\varphi^*\bar{g}$ respectively, $\iota:S\to M$ is the natural inclusion map, and
\begin{equation}
\eta_a{}^b=\frac{1}{2}\sqrt{-\det g}g^{b\beta}\epsilon_{a\beta\mu\nu}dx^\mu\wedge dx^\nu.
\end{equation}
Although the 4D isometric matching reference $\varphi$ is by definition smooth on a neighborhood of $S$, the value of $E(S,N,\varphi)$ depends only on the 1-jet of $\varphi$ (see \cite{LY}). That is to say,
$$E(S,N,\varphi_1)=E(S,N,\varphi_2)$$
for any two references $\varphi_1$ and $\varphi_2$ with
$$\varphi_1=\varphi_2\ \mbox{and}\ d\varphi_1=d\varphi_2$$
on $S$.

Consider the axisymmetric Kerr-like spacetime $(M,g)$ in the Boyer-Lindquist coordinates \cite{SCLN2}:
\begin{equation}\label{eqn-kerr-like}
g=Fdt^2+2Gdtd\phi+Hd\phi^2+R^2dr^2+\Sigma^2d\theta^2,
\end{equation}
where the components $F,G,H,R,\Sigma$ are functions of $r$, $\theta$. For the Kerr spacetime, they are
\begin{equation}\label{eqn-kerr}
\left\{
\begin{array}{l}
F=-\frac{\Delta-a^2\sin^2\theta}{\Sigma^2}\\
G=-\frac{2mar\sin^2\theta}{\Sigma^2}\\
H\Sigma^2=\sin^2\theta((r^2+a^2)^2-\Delta a^2\sin^2\theta)\\
R^2\Delta=\Sigma^2=r^2+a^2\cos^2\theta\\
\Delta=r^2-2mr+a^2.
\end{array}\right.
\end{equation}
Here, we suppose $0\leq a\leq m$. Then the 4D isometric matching equation for an axially symmetric embedding:
\begin{equation}\label{eqn-axis-embedding}
\left\{\begin{array}{l}T=T(t,r,\theta)\\
X=\rho(t,r,\theta)\cos(\phi+\Phi(t,r,\theta))\\
Y=\rho(t,r,\theta)\sin(\phi+\Phi(t,r,\theta))\\
Z=Z(t,r,\theta)
\end{array}\right.
\end{equation}
into $\R^{1,3}$ of a constant radius surface
$$S(t_0,r_0)=\{(t,r,\theta,\phi)\ |\ t=t_0,r=r_0 \}$$
is indeed explicitly solvable (see \cite{SCLN2}). With 4D isometric matching, there are still two degrees of freedom. These can be taken as the two functions $x$ and $y$ on the 2-surface. The critical value of the QLE is corresponding to the solution $x,y$ of the energy optimization. Set
\begin{equation}\label{eqn-x-y}
\left\{\begin{array}{l}x(\theta)=T_r(t_0,r_0,\theta)\\
y(\theta)=T_\theta(t_0,r_0,\theta)
\end{array}\right.
\end{equation}
and
\begin{equation}
N=(\varphi^{-1})_*\frac{\p}{\p T}.
\end{equation}
Then,
\begin{equation}\label{eqn-CNT-x-y}
E(x,y):=E(S,N, \varphi)=\frac{1}{4}\int_0^{\pi}\mathfrak{B}(x,y)d\theta
\end{equation}
with (see Appendix of \cite{LT})
\begin{equation}\label{eqn-CNT-B}
\begin{split}
\mathfrak{B}(x,y)=&-\frac{\alpha(H\Sigma^2)_r}{2\sqrt H R^2\Sigma^2}-\sqrt H\left(\frac{H_{\theta\theta}-2l}{\beta}+\frac{R_\theta xy}{R\alpha}-\frac{xy^3\beta+H_\theta\alpha\Sigma^2}{l\alpha\beta\Sigma}\Sigma_\theta\right)\\
&+\frac{\sqrt Hyx_\theta}{\alpha}+\frac{\sqrt Hy(H_\theta\alpha-xy\beta)}{l\alpha\beta}y_\theta,
\end{split}
\end{equation}
where
\begin{equation}\label{eqn-a-b-l}
\left\{\begin{array}{l}\alpha=\sqrt{x^2\Sigma^2+R^2l}\\
\beta=\sqrt{-H_\theta^2+4Hl}\\
l=y^2+\Sigma^2.
\end{array}\right.
\end{equation}

The Euler-Lagrange equation of $E(x,y)$ is
\begin{equation}\label{eqn-critical}
\left\{
\begin{split}
y_\theta=&-\frac{(\Sigma^2H)_r}{2HR^2}x-\frac{\Sigma H_\theta-2H\Sigma_\theta}{2H\Sigma}y\\
x_\theta=&\frac{R_\theta}{R}x+\left(\frac{(\Sigma^2H)_r}{2H\Sigma^2}-\frac{\alpha\beta+xyH_\theta}{2Hl}\right)y
\end{split}\right.
\end{equation}
which has an obvious solution $x\equiv y\equiv 0$. It was shown in \cite{LT} that
\begin{equation}
E(0,0)=\mathfrak{m}_{\mathrm{BY}}(S).
\end{equation}
Here
$$\mathfrak{m}_{\mathrm{BY}}(S):=\frac{1}{8\pi}\int_S (k_0-k)dV_s$$
is the Brown-York mass of $S$ \cite{BY}, where $k_0$ and $k$ are the mean curvatures of $S$ and $\bar{S}$ respectively. Note that for the Brown-York definition, $\bar{S}$ is the isometric embedding of $S$ into $\mathbb{R}^3$. Here, if the constant $t_{0}$, $r_{0}$ surface can be embedded into the constant $T$ hyperplane of $\mathbb{R}^{1,3}$, then the critical value of the QLE is equal to the Brown-York mass, as shown in \cite{LT}.

Let $K(a,r_0,\theta)$ be the Gaussian curvature, $k(a,r_0,\theta)$ be the mean curvature of the constant radius surface with $r=r_0$, and $k_0(a,r_0,\theta)$ be the mean curvature of $\bar{S}$. 
We prove the following theorem in the next two sections:
\begin{thm}\label{thm-main}
\begin{itemize}
\hspace{50pt}
\item[(1)] When $r>r_{+}(a):=m+\sqrt{m^2-a^2}$, $K(a,r,\theta)>0$ for any $\theta$ if and only if
\begin{equation}
r>r_k(a).
\end{equation}
Here, $r_k(a)$ is the unique real root of the cubic polynomial: $r^3+a^2r-6a^2m$,
\begin{equation}\label{eqn-rk}
\begin{split}
r_k(a)=&-\frac{a^2}{\sqrt[3]3\sqrt[3]{27a^2m+\sqrt 3\sqrt{243a^4m^2+a^6}}}\\
&+\frac{\sqrt[3]{27a^2m+\sqrt 3\sqrt{243a^4m^2+a^6}}}{\sqrt[3] 9};
\end{split}
\end{equation}
\item[(2)] when $r>r_{+}(a)$, $k_0(a,r,\theta)-k(a,r,\theta)>0$ for any $\theta$ if and only if
    \begin{equation}
    r>\sqrt 3 a.
    \end{equation}
\end{itemize}
\end{thm}
Note that $r_k(a)$ is an increasing function of $a\in [0,m]$ with $r_k(0)=0$ and  $r_k(m)\fallingdotseq1.63437m$. So that (1) of Theorem \ref{thm-main} improves a result of [\cite{LT}, Theorem 3.1]. Moreover, note that the graphs of $r=r_k(a)$, $r=r_+(a)$ and $r=\sqrt3 a$ intersect at the point $(a,r)=(\sqrt{3}m/2,3m/2)$ (see Fig.\ \ref{fig:1}). The triangle-like region $\mathfrak{A}$ is defined by
$$\mathfrak A:=\{(r,a)\ |\ r_k(a)< r < \sqrt{3}a,\ a\in [\sqrt{3}m/2,m]\}.$$
As long as $(a,r)$ is in this region, the constant radius surface $S(t,r)$ is outside the outer horizon with positive Gaussian curvature but $k_0-k<0$ for some $\theta$. Therefore there is no known theorem to guarantee the positivity of the QLE. Later in Sec. 6, we show numerical results which imply positivity and the monotonically decreasing property.

By setting $t=m\tilde{t}$, $r=m\tilde{r}$, $a=m\tilde{a}$ and preserving $\theta$ and $\phi$ we have $ds^2=m^2d\tilde{s}^2$, where 
\begin{equation}
\begin{split}
d\tilde{s}^2&=-\frac{\tilde{\Delta}}{\tilde{\Sigma}^2}[d\tilde{t}-\tilde{a}\sin\theta d\phi]^2+\frac{\sin^2\theta}{\tilde{\Sigma}^2}[(\tilde{r}^2+\tilde{a}^2)d\phi-\tilde{a}d\tilde{t}]^2\\
&+\frac{\tilde{\Sigma}^2}{\tilde{\Delta}}d\tilde{r}^2+\tilde{\Sigma}^2d\theta^2,\nonumber
\end{split}
\end{equation}
$\tilde{\Delta}=\tilde{r}^{2}-2\tilde{r}+\tilde{a}^2$, $\tilde{\Sigma}^2=\tilde{r}^{2}+\tilde{a}^2\cos^2\theta$. So, without loss of generality, we will simply assume that $m=1$ in the proofs in the following sections.

\section{Gaussian curvature}
In the previous section, we introduced the QLE for the surface with constant $t,r$ in the Kerr spacetime, and the critical value of the QLE is the Brown-York mass, which involves the 2-surface isometric embedding into $\mathbb{R}^3$. It can be solved explicitly in the Kerr case for positive Gaussian curvature. From \cite{LT}, the Gaussian curvature $K$ of the constant radius surface $S(t,r)$ for the Kerr-like spacetime \eqref{eqn-kerr-like} is
\begin{equation}\label{eqn-K-1}
K=\frac{H(4\Sigma^4+H_\theta(\Sigma^2)_\theta-2H_{\theta\theta}\Sigma^2)-\Sigma^2(-H_\theta^2+4H\Sigma^2)}{4\Sigma^4H^2}.
\end{equation}
It can be simplified as
\begin{equation}\label{eqn-K-2}
\begin{split}
K=&\frac{1}{4}\frac{H(H_\theta(\Sigma^2)_\theta-2H_{\theta\theta}\Sigma^2)+\Sigma^2H_\theta^2}{\Sigma^4H^2}\\
=&\frac{H_\theta(H\Sigma^2)_\theta-2H_{\theta\theta}H\Sigma^2}{4\Sigma^4H^2}\\
=&-\frac{1}{4}\left(H_\theta\left(\frac{1}{H\Sigma^2}\right)_\theta+2H_{\theta\theta}\frac{1}{H\Sigma^2}\right)\\
=&-\frac{1}{4H_\theta}\left(H_\theta^2\left(\frac{1}{H\Sigma^2}\right)_\theta+\left(H_\theta^2\right)_\theta \frac{1}{H\Sigma^2}\right)\\
=&\frac{\left(1-\frac{H_\theta^2}{4H\Sigma^2}\right)_\theta}{H_\theta}.
\end{split}
\end{equation}
Let $\sigma=\sin^2\theta$. Then from \eqref{eqn-kerr} we have
\begin{equation}\label{eqn-HS}
H\Sigma^2=\sigma((a^2+r^2)^2-\Delta a^2\sigma), \quad H=\frac{\sigma((a^2+r^2)^2-\Delta a^2\sigma)}{(a^2+r^2)-a^2\sigma}.
\end{equation}
By direct computation, we have
\begin{equation}\label{eqn-H-1}
H_\sigma=\frac{(a^2+r^2)^3-2(a^2+r^2)\Delta a^2\sigma+\Delta a^4\sigma^2}{\left((a^2+r^2)-a^2\sigma\right)^2}
\end{equation}
and
\begin{equation}\label{eqn-H-2}
H_{\sigma\sigma}=\frac{4r(a^2+r^2)^2a^2}{\left((a^2+r^2)-a^2\sigma\right)^3}.
\end{equation}

We are now ready to prove (1) of Theorem \ref{thm-main}.
\begin{proof}[Proof of (1) in Theorem \ref{thm-main}] Note that
\begin{equation}
H_\theta=2H_\sigma \sin\theta \cos\theta,
\end{equation}
\begin{equation}
\begin{split}
H_{\theta\theta}=&4H_{\sigma\sigma}\sin^2\theta \cos^2\theta+2H_\sigma \cos^2\theta-2H_\sigma \sin^2\theta\\
=&4H_{\sigma\sigma}\sigma(1-\sigma)+2H_\sigma(1-2\sigma),
\end{split}
\end{equation}
and by \eqref{eqn-HS}
\begin{equation}
\begin{split}
(H\Sigma^2)_\theta=&2(H\Sigma^2)_\sigma\sin\theta\cos\theta\\
=&2((a^2+r^2)^2-2\Delta a^2\sigma)\sin\theta\cos\theta.
\end{split}
\end{equation}
Substituting these into the second equality of \eqref{eqn-K-2}, we have
\begin{equation}\label{eqn-K-3}
\begin{split}
K=&\frac{H_\theta(H\Sigma^2)_\theta-2H_{\theta\theta}H\Sigma^2}{4\Sigma^4H^2}\\
=&\frac{f(\sigma)}{[(a^2+r^2)^2-\Delta a^2\sigma]^2},
\end{split}
\end{equation}
where
\begin{equation}\label{eqn-f}
f(\sigma)=H_\sigma[(a^2+r^2)^2-\Delta a^2]-2H_{\sigma\sigma}[(a^2+r^2)^2-\Delta a^2\sigma](1-\sigma).
\end{equation}
Then, by \eqref{eqn-H-2},
\begin{equation}\label{eqn-f-1}
\begin{split}
f'=&H_{\sigma\sigma}[3(a^2+r^2)^2+\Delta a^2-4\Delta a^2\sigma]\\
&-2H_{\sigma\sigma\sigma}[(a^2+r^2)^2-\Delta a^2\sigma](1-\sigma)\\
=&\frac{4r(a^2+r^2)^2a^2 U(\sigma)}{\left((a^2+r^2)-a^2\sigma\right)^4},
\end{split}
\end{equation}
where
\begin{equation}\label{eqn-U}
\begin{split}
U(\sigma)=&\left((a^2+r^2)-a^2\sigma\right)[3(a^2+r^2)^2+\Delta a^2-4\Delta a^2\sigma]\\
&-6a^2[(a^2+r^2)^2-\Delta a^2\sigma](1-\sigma).
\end{split}
\end{equation}
Note that $U''=-4\Delta a^4<0$,
\begin{equation}\label{eqn-U0}
\begin{split}
U(0)=&(a^2+r^2)(3(a^2+r^2)^2+\Delta a^2)-6a^2(a^2+r^2)^2\\
\geq&3(a^2+r^2)^2(r^2-a^2)\\
\geq&0
\end{split}
\end{equation}
since $r>r_+\geq a$, and
\begin{equation}\label{eqn-U1}
U(1)=3r^2((a^2+r^2)^2-\Delta a^2)>0.
\end{equation}
So,
\begin{equation}
U(\sigma)\geq0
\end{equation}
for any $\sigma\in [0,1]$. Then, by \eqref{eqn-f-1}, $f$ is increasing in $[0,1]$. Moreover, note that
\begin{equation}
f(0)=(a^2+r^2)r(r^3+a^2r-6a^2),
\end{equation}
therefore by \eqref{eqn-K-3}, we have
\begin{equation}
\inf_\theta K(\theta,a,r)>0
\end{equation}
when $r>r_+(a)$ and
\begin{equation}
r^3+a^2r-6a^2>0.
\end{equation}
Conversely
\begin{equation}\label{eqn-K0}
K(0)=\frac{r(r^3+a^2r-6a^2)}{(a^2+r^2)^3},
\end{equation}
so that if $K>0$, we must have $r>r_k(a)$. This completes the proof of (1) in Theorem \ref{thm-main}.
\end{proof}

\section{Difference of mean curvatures}
In this section, we prove (2) of Theorem \ref{thm-main}. 
\begin{proof}
The integrand of the critical quasi-local energy is $k_{0}-k$. The mean curvature $k$  \cite{LT} of the constant radius surface $S(t,r)$ for Kerr-like spacetime \eqref{eqn-kerr-like} is
\begin{equation}\label{eqn-k}
k=\frac{(H\Sigma^2)_r}{2H\Sigma^2R},
\end{equation}
and the mean curvature $k_0$ of $S(t,r)$ isometrically embedded into $\R^3$ (if such an embedding exists) is
\begin{equation}\label{eqn-k0}
k_0=\frac{H_\theta(\Sigma^2)_\theta+4\Sigma^4-2H_{\theta\theta}\Sigma^2}{2\Sigma^3(-H_\theta^2+4H\Sigma^2)^{1/2}}.
\end{equation}
Let $L=1-\frac{H_\theta^2}{4H\Sigma^2}$, from \eqref{eqn-K-1} we obtain \begin{equation}
H_\theta(\Sigma^2)_\theta+4\Sigma^4-2H_{\theta\theta}\Sigma^2=4H\Sigma^4K+4\Sigma^4L.
\end{equation} 
Substituting this into \eqref{eqn-k0}, we have
\begin{equation}\label{eqn-k0-1}
k_0=\sqrt{\frac{H}{L}}K+\sqrt{\frac{L}{H}}.
\end{equation}
From the last equality of \eqref{eqn-K-2} for $m=1$ and $\sigma=\sin^2\theta$ we have
\begin{equation}\label{eqn-L}
L_\sigma=H_\sigma K.
\end{equation}
From \eqref{eqn-k}, when $r>r_+(a)$, by direct computation we have
\begin{equation}\label{eqn-k-2}
k=\frac{(2r(r^2+a^2)-(r-1)a^2\sigma)\sqrt \Delta}{((r^2+a^2)^2-\Delta a^2\sigma)\sqrt{r^2+a^2-a^2\sigma}}.
\end{equation}

To prove the necessary condition of Theorem \ref{thm-main} (2), it is sufficient to consider the north pole $\sigma=0$. By \eqref{eqn-L} we have
\begin{equation}
\lim_{\sigma\to 0}\frac{L(\sigma)}{H(\sigma)}=\frac{L_\sigma(0)}{H_\sigma(0)}=K(0),
\end{equation}
so that $k_{0}(0)=K(0)/\sqrt{K(0)}+\sqrt{K(0)}=2\sqrt{K(0)}$. From \eqref{eqn-k-2} the mean curvature is
\begin{equation}
k(0)=\frac{2r\sqrt\Delta}{(r^2+a^2)^\frac32}.
\end{equation}
Consequently, when $r>r_+(a)$ and $r>r_k(a)$, we have
\begin{equation}\label{eqn-k-k0-0}
k_0^2(0)-k^2(0)=4K(0)-k^2(0)=\frac{8r(r^2-3a^2)}{(r^2+a^2)^3}.
\end{equation}
Therefore $r>\sqrt 3 a$ is necessary for $k_0-k>0$.

Conversely, note that the three curves $r=r_+(a)$, $r=r_k(a)$ and $r=\sqrt 3a$ intersect at the point $(a,r)=(\sqrt 3/2,3/2)$. It is clear that
\begin{equation}
\sqrt 3a<r_k(a)< r_+(a)
\end{equation}
when $a\in (0,\sqrt3/2)$, and
\begin{equation}
r_+(a)<r_k(a)<\sqrt 3a
\end{equation}
for $a\in (\sqrt 3/2,1)$. So, by (1) of Theorem \ref{thm-main}, when $r>r_+(a)$ and $r>\sqrt 3 a$, we obtain $K>0$ for $a\in(0,1)$. From \eqref{eqn-k0-1} it is not difficult to see that $\left(K\sqrt{H/L}-\sqrt{L/H}\right)^2\geq0$ implies 
\begin{equation}
k_{0}\geq 2\sqrt{K}
\end{equation}
for $K\geq0$. Hence, by \eqref{eqn-K-3}
\begin{equation}\label{eqn-k0-k}
k_0^2-k^2\geq 4K-k^2=\frac{4f-g}{[(a^2+r^2)^2-\Delta a^2\sigma]^2}.
\end{equation}
where $f$ is given in \eqref{eqn-f} and
\begin{equation}
\begin{split}
g=&\frac{(2r(r^2+a^2)-(r-1)a^2\sigma)^2\Delta}{r^2+a^2-a^2\sigma}\\
=&\frac{((r+1)(r^2+a^2)+(r-1)(r^2+a^2-a^2\sigma))^2\Delta}{r^2+a^2-a^2\sigma}.
\end{split}
\end{equation}
Let
\begin{equation}
h=4f-g.
\end{equation}
By direct computation,
\begin{equation}
\begin{split}
g'=&\frac{\Delta a^2[(r+1)^2(r^2+a^2)^2-(r-1)^2(r^2+a^2-a^2\sigma)^2]}{(r^2+a^2-a^2\sigma)^2}\\
\leq&\frac{\Delta a^2[(r+1)^2(r^2+a^2)^2-(r-1)^2r^4]}{(r^2+a^2-a^2\sigma)^2}\\
\leq &\frac{\Delta a^2[(r+1)^2(r^2+1)^2-(r-1)^2r^4]}{(r^2+a^2-a^2\sigma)^2}\\
=&\frac{\Delta a^2(4r^5+2r^4+4r^3+3r^2+2r+1)}{(r^2+a^2-a^2\sigma)^2}\\
\leq& \frac{16\Delta a^2r^5}{(r^2+a^2-a^2\sigma)^2}\\
\leq&\frac{16\Delta a^2r^5(r^2+a^2)^2}{(r^2+a^2-a^2\sigma)^4}
\end{split}
\end{equation}
since $r>r_+(a)\geq 1\geq a$.
Then, by \eqref{eqn-f-1},
\begin{equation}\label{eqn-h-1}
\begin{split}
h'=&4f'-g'\\
\geq& \frac{16a^2(r^2+a^2)^2rU}{(r^2+a^2-a^2\sigma)^4}-\frac{16\Delta a^2r^5(r^2+a^2)^2}{(r^2+a^2-a^2\sigma)^4}\\
=&\frac{16a^2(r^2+a^2)^2r V}{(r^2+a^2-a^2\sigma)^4},
\end{split}
\end{equation}
where
\begin{equation}
V=U-\Delta r^4
\end{equation}
with $U$ given in \eqref{eqn-U}. Note that
\begin{equation}
V''(\sigma)=-4\Delta a^2<0.
\end{equation}
Moreover, by \eqref{eqn-U0}
\begin{equation}
\begin{split}
V(0)=&U(0)-\Delta r^4\\
\geq&3(r^2+a^2)^2(r^2-a^2)-\Delta r^4\\
\geq&2r^4(r^2-a^2)\\
\geq& 0
\end{split}
\end{equation}
since $\Delta=r^2-2r+a^2\leq r^2-a^2$, and by \eqref{eqn-U1}
\begin{equation}
\begin{split}
V(1)= &U(1)-\Delta r^4\\
=&r^2(3((a^2+r^2)^2-\Delta a^2)-\Delta r^2)\\
\geq&r^2(3(a^2+r^2)^2-3\Delta a^2-3\Delta r^2)\\
=&6r^3(r^2+a^2)\\
>&0.
\end{split}
\end{equation}
Hence $V(\sigma)\geq 0$ for any $\sigma\in [0,1]$. By \eqref{eqn-h-1}, $h$ is increasing on $[0,1]$. By \eqref{eqn-k-k0-0}, we know that $h(0)>0$ when $r> r_+(a)$ and $r>\sqrt 3a$. So $h>0$ on $[0,1]$ and hence $k_0-k>0$ for any $\theta$.
\end{proof}


\section{The global minimum of Kerr quasi-local energy}
In \cite{LY}, we found that the CNT quasi-local energy is closely related to
Wang-Yau's energy $E_{\rm{WY}}$ \cite{WaYa}. More precisely, let $\varphi_0:S\to \R^{1,3}$ be an isometric embedding and $\tau$ be the time component of the embedding, and suppose the mean curvature vector $H$ of $S$ in $M$ is spacelike. Then,
\begin{equation}\label{eqn-CNT-WY}
\begin{split}
E(S,N_0,\varphi)=&E_{\rm{WY}}(S,\tau)\\
:=&\frac{1}{8\pi}\int_{\bar{S}}\left(-\sqrt{1+\|\nabla\tau\|^2}\vv<\bar{H},\bar{e}_1>+\vv<\bar{\nabla}_{-\nabla\tau}\bar{e}_1,\bar{e}_0>\right)dV_{\bar{S}}\\
&-\frac{1}{8\pi}\int_{S}\left(-\sqrt{1+\|\nabla\tau\|^2}\vv<H, e_1>+\vv< \nabla_{-\nabla\tau}e_1,e_0>\right)dV_S.\\
\end{split}
\end{equation}
Here $\bar{S}=\varphi_0(S)$, $H$ and $\bar{H}$ are the mean curvature vector of $S$ and $\bar{S}$ respectively; $e_0$ is a future-directed time-like vector such that
\begin{equation}\label{eqn-e-0}
\vv<H,e_0>=-\frac{\Delta\tau}{\sqrt{1+\|\nabla\tau\|^2}};
\end{equation}
$e_1$ is orthogonal to $e_0$ and $S$, and pointing outside if $S$ encloses a domain $\Omega$; $\bar{e}_1$ is pointing outside and orthogonal to $\bar{S}$ and $\frac{\p}{\p T}$; $\bar{e}_0$ is a future-directed time-like vector which is orthogonal to $\bar{S}$ and $\bar{e}_1$;
\begin{equation}\label{eqn-N-0-1}
N_0=\sqrt{1+\|\nabla \tau\|^2}e_0-\nabla\tau;
\end{equation}
and
$\varphi$ is a 4D isometric matching extension of $\varphi_0$ \cite{LY}.
$\nabla \tau$ and $\Delta\tau$ here mean the gradient and Laplacian of $\tau$ with respect to the induced metric on $S$.

By direct computation [\cite{LY} Appendix], the first equation of \eqref{eqn-critical} corresponds to
\begin{equation}
N_0=(\varphi^{-1})_*\frac{\p}{\p T}.
\end{equation}
Let
\begin{equation}
E(y):=E(x,y)
\end{equation}
with $x$ decided by $y$ from the first equation of \eqref{eqn-critical}. Then, by the uniqueness of the isometric embedding into $\R^3$ and  the relation \eqref{eqn-CNT-WY}, we obtain
\begin{equation}\label{eqn-CNTWY}
E(y)=E_{\mathrm{WY}}(S, \tau),
\end{equation}
where $\tau$ depends only on $\theta$ and $y=\frac{d\tau}{d\theta}$.

Furthermore, by combining Theorem \ref{thm-main} and Theorem 3 of Chen-Wang-Yau \cite{CWY}, and \eqref{eqn-CNTWY} (other minimizing properties of the Wang-Yau QLE can be found in \cite{MTX,MT}), we have the following corollary:
\begin{cor}\label{cor-main}
\hspace{50pt}
\begin{enumerate}
\item When $r>r_h(a)$,
$$E_{\mathrm WY}(\tau)\geq E_{\mathrm WY}(0),$$
 where $\tau$ is only a function of $\theta$ such that $Hd\phi^2+(\Sigma^2+\tau_\theta^2)d\theta^2$ has positive Gaussian curvature;
\item when $r>r_h(a)$,
$$E(y)\geq E(0),$$
where $y$ is a function of $\theta$ such that $Hd\phi^2+(\Sigma^2+y^2)d\theta^2$ has positive Gaussian curvature.
\end{enumerate}
Here
\begin{equation}\label{eqn-rh}
r_h(a)=\left\{\begin{array}{ll}r_+(a)& a\leq \sqrt{3}m/2\\
\sqrt 3 a&a>\sqrt{3}m/2.
\end{array}\right.
\end{equation}
\end{cor}

This implies that the critical value of the Kerr QLE is the global minimum with respect to $y$. An interesting question is whether Corollary \ref{cor-main} is satisfied in the triangle-like region $\mathfrak A$. 

\subsection*{Remark}
In order to use Chen-Wang-Yau's result \cite{CWY}, one should check that $\tau=0$ is a critical point of $E_{\mathrm{WY}}$. According to \cite{MTX,WaYa}, $\tau=0$ is a critical point of $E_{\mathrm{WY}}$ if and only if $\mathrm{div}_{S}W=0$, where $W$ is the vector field dual to $\alpha_{H}$, which is the connection one form defined by $\alpha_{H}=\langle\nabla^{M} e_{1},e_{0}\rangle_{g}$. Here $e_{1}$ is the spacelike unit vector normal to $S$ and $e_{0}$ is the future directed timelike vector normal to the hypersurface $\Omega$. By definition $\alpha_{H}(\nabla\tau)=\langle W,\nabla\tau\rangle_{g}$, we have
\begin{equation}
\alpha_{H}(\nabla\tau)=\langle\nabla^{M} e_{1},e_{0}\rangle_{g}(\nabla\tau)=-\omega^{0}{}_{1}(\nabla\tau)
\end{equation}
which implies $W=-\omega^{0}{}_{12}e_{2}-\omega^{0}{}_{13}e_{3}$, where $\omega^{a}{}_{bc}$ is the connection coefficient corresponding to the orthonormal frame $e_{0}=\beta(\partial_{t}+\omega\partial_{\phi})$, $e_{1}=(1/R)\partial_{r}$, $e_{2}=(1/\Sigma)\partial_{\theta}$, $e_{3}=(1/\sqrt{H})\partial_{\phi}$. Here $e_{0}$ is the locally nonrotating observer (see the next section) with $\beta=\sqrt{H}/\sqrt{G^2-FH}$ and angular velocity $\omega=-G/H$. 

The divergence of $W$ is
\begin{equation}
\mathrm{div}_{S}W=\frac{1}{\sqrt{\sigma}}[\partial_{\theta}(\sqrt{\sigma}W^{\theta})+\partial_{\phi}(\sqrt{\sigma}W^{\phi})],
\end{equation}
where $\sqrt{\sigma}=\Sigma\sqrt{H}$ is the determinant of the metric induced on $S$ and $W^{\theta}=-\omega^{0}{}_{12}/\Sigma$, $W^{\phi}=-\omega^{0}{}_{13}/\sqrt{H}$. Note that the metric components of the Kerr spacetime in Boyer-Lindquist coordinates depend only on $r,\theta$ so that the connection coefficients are independent of $\phi$. This implies that $\partial_{\phi}(\sqrt{\sigma}W^\phi)$ vanishes. It is not difficult to find that $\omega^{0}{}_{12}=0$ and consequently $\mathrm{div}_{S}W=0$, which implies that $\tau=0$ is a critical point of $E_{\mathrm{WY}}$.

\section{Numerical results}
By 4D isometric matching, the displacement vector $N$ is identical to the timelike Killing vector field of the reference spacetime: $\varphi_{*}N=\partial_{T}$, and $g(N,N)=\bar{g}(\partial_{T},\partial_{T})=-1$ on the 2-surface $S$. In the physical spacetime, although the timelike vector $\partial_{t}$ becomes spacelike inside the ergosphere, $N$ is still timelike, with components [\cite{SCLN2}, Eq.(47)]
\begin{equation}
N^t=\frac{\sqrt{H}\alpha}{\sqrt{-g}},\quad N^r=-\frac{x}{R^2},\quad N^\theta=-\frac{y}{\Sigma^2},\quad N^\phi=\frac{-G\alpha}{\sqrt{-g}\sqrt{H}}.
\end{equation}
It is the \emph{locally nonrotating observer} [\cite{MTW}, p.896], or the so-called ZAMO (zero angular momentum observer) at the critical point $(x,y)=(0,0)$:
\begin{equation}
N^t=\frac{\sqrt{H}}{\sqrt{G^2-FH}},\quad N^r=N^\theta=0,\quad N^{\phi}=\frac{-G}{\sqrt{G^2-FH}\sqrt{H}}.
\end{equation}
Note that this timelike $N$ has the form $\beta\partial_{\tau}$ \cite{Frolov}, where $\partial_{\tau}=\partial_{t}+\omega\partial_{\phi}$ with the angular velocity $\omega=-G/H$ and $\partial_{\tau}$ is timelike in the regions outside the outer horizon and inside the inner horizon.

From the result of \cite{SCLN2}, it follows that for constant $a$, the QLE is monotonically decreasing from $r=2m$ to large $r$. Here we analyze the QLE for the region $r\leq 2m$, which has some part of the surface inside the ergosphere. We also recheck the results for $r\geq2m$. 

These results (see Figs.\ 2--4) imply that: (i) for increasing $a$, the QLE is \emph{decreasing} and (ii) it is monotonically decreasing with respect to $r$. Concerning (i), if we consider the irreducible mass $M_{\mathrm{ir}}$ [\cite{MTW}, p.890]
\begin{equation}
a=2M_{\mathrm{ir}}\sqrt{1-\frac{M_{\mathrm{ir}}^2}{m^2}},
\end{equation}
then the QLE is increasing with respect to $M_{\mathrm{ir}}$, where $M_{\mathrm{ir}}\leq m$ and the equality holds for the Schwarzschild limit ($a=0$) and $M_{\mathrm{ir}}=m/\sqrt{2}$ for $a=m$. It is not difficult to see that $M_{\mathrm{ir}}$ is decreasing with respect to $a$. So that $0\leq a\leq m$ respectively corresponds to
$m\geq M_{\mathrm{ir}}\geq \sqrt{2}m/2$. 

For the region $r\leq 2m$, it should be considered into two parts:
\begin{equation}
\left\{\begin{split}
r>r_{+}:&\quad 0\leq a\leq \sqrt{3}m/2,\quad m\geq M_{\mathrm{ir}}\geq \sqrt{3}m/2, \\ 
r\geq r_{k}:&\quad \sqrt{3}m/2\leq a\leq m,\quad \sqrt{3}m/2\geq M_{\mathrm{ir}}\geq\sqrt{2}m/2.
\end{split}\right.
\end{equation}
Note that for $a=0$, the Schwarzschild case, the irreducible mass $M_{\mathrm{ir}}=m$ and its quasi-local energy at the horizon $r=2m$ is $E=2m$.

The Figs.\ 2--4 are the plots by the approximation of the boundary integration \eqref{eqn-CNT-x-y} at $(x,y)=(0,0)$: $E\approx\frac{1}{4}\sum\mathfrak{B}_{n}\Delta\theta$, where we pick $\Delta\theta=0.001$, and $\mathfrak{B}_{n}=\mathfrak{B}|_{\theta=n\Delta\theta}$ starting from $n=0$ to the last step which is $\theta=\pi$. Each step has the interval $\Delta\theta=0.001$. Note that $\mathfrak{B}(\theta\rightarrow0,\pi)=0$.

\begin{figure}[!htb]
	\centering
	\begin{subfigure}{0.5\textwidth}
		\centering
		\includegraphics[width=0.95\linewidth, height=0.27\textheight]{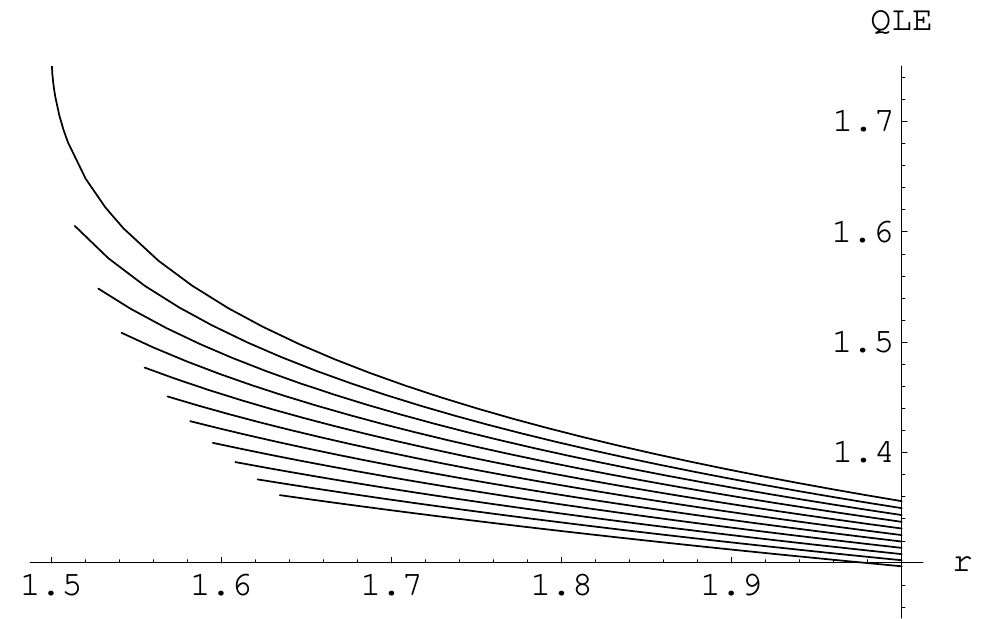}
		\caption{\scriptsize $\sqrt{3}/2\leq a\leq1$.\\$a=\sqrt{3}/2$ is the top curve;\\$a=1$ is the bottom curve.}\label{fig:1a}		 
	\end{subfigure}%
		\begin{subfigure}{0.5\textwidth}
		\centering
		\includegraphics[width=0.95\linewidth, height=0.27\textheight]{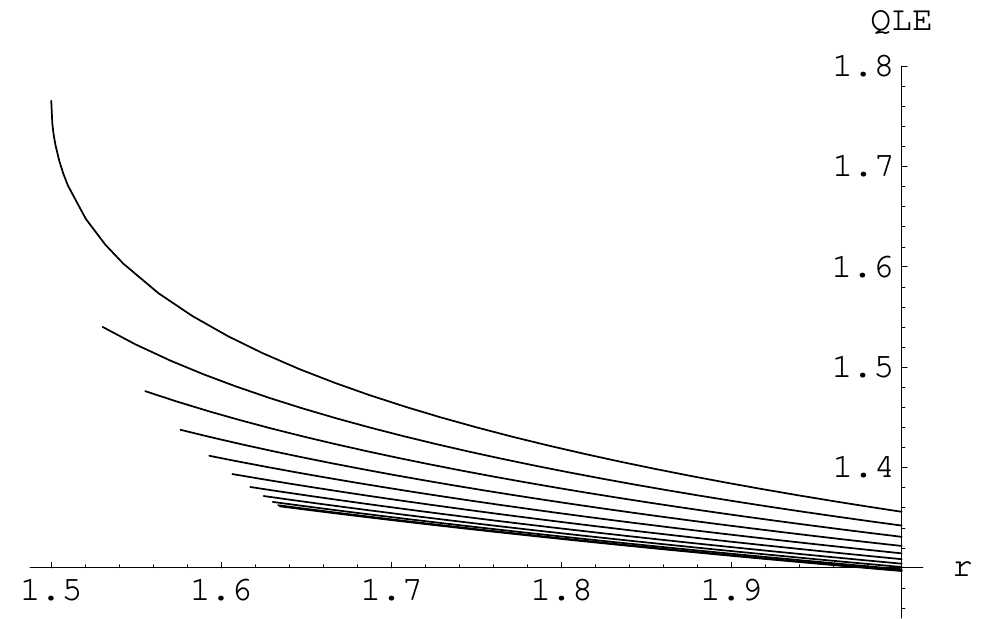}
		\caption{\scriptsize $\sqrt{2}/2\leq M_{\mathrm{ir}}\leq\sqrt{3}/2$.\\ $M_{\mathrm{ir}}=\sqrt{3}/2$ is the top curve;\\$M_{\mathrm{ir}}=\sqrt{2}/2$ is the bottom curve.}\label{fig:1b}
	\end{subfigure}
	\caption{\hspace{55pt}$r_{k}\leq r\leq 2$, $m=1$.\\The curves are for constant $a$ and constant $M_{\mathrm{ir}}$ from the bottom curve to the top with the interval $0.1\Delta a$ and $0.1\Delta M_{\mathrm{ir}}$, where $\Delta a=1-\sqrt{3}/2$ and $\Delta M_{\mathrm{ir}}=\sqrt{3}/2-\sqrt{2}/2$. They show that (i) the QLE is decreasing with respect to $a$ and increasing with respect to $M_{\mathrm{ir}}$. (ii) the QLE is monotonically decreasing with respect to $r$.}\label{fig:3}
\end{figure}

\begin{figure}[!htb]
	\centering
	\begin{subfigure}{0.5\textwidth}
		\centering
		\includegraphics[width=0.95\linewidth, height=0.27\textheight]{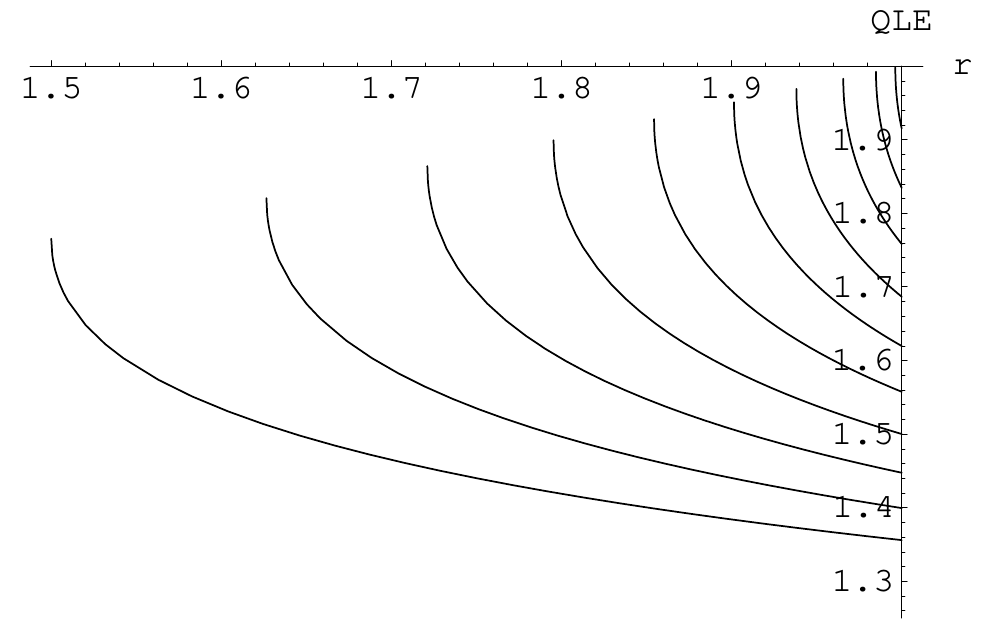}
		\caption{\scriptsize$0\leq a\leq\sqrt{3}/2$.\\$a=\sqrt{3}/2$ is the bottom curve;\\$a=0$ is the point $\mathrm{QLE}=2m$.}\label{fig:2a}		 
	\end{subfigure}%
		\begin{subfigure}{0.5\textwidth}
		\centering
		\includegraphics[width=0.95\linewidth, height=0.27\textheight]{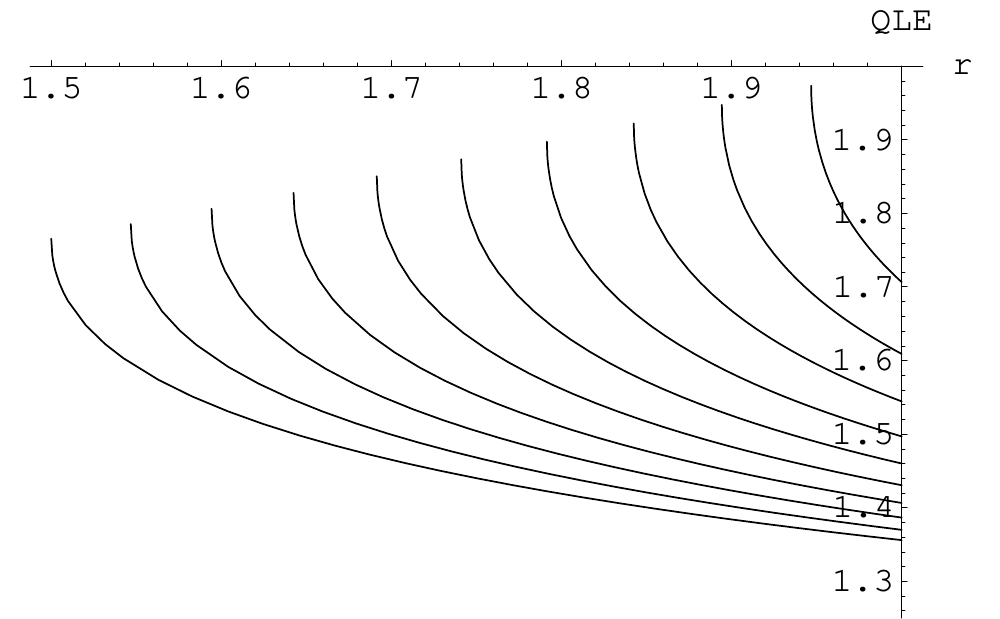}
		\caption{\scriptsize $\sqrt{3}/2\leq M_{\mathrm{ir}}\leq 1$.\\$M_{\mathrm{ir}}=\sqrt{3}/2$ is the bottom curve;\\$M_{\mathrm{ir}}=1$ is the point $\mathrm{QLE}=2m$.}\label{fig:2b}
	\end{subfigure}
	\caption{\hspace{55pt}$r_{+}< r\leq2$, $m=1$.\\The curves are for constant $a$ and the irreducible mass $M_{\mathrm{ir}}$ from the bottom curve to the top with the interval $0.1\Delta a$ and $0.1\Delta M_{\mathrm{ir}}$, where $\Delta a=\sqrt{3}/2$ and $\Delta M_{\mathrm{ir}}=1-\sqrt{3}/2$. They show that (i) the QLE is decreasing with respect to $a$ and increasing with respect to $M_{\mathrm{ir}}$. (ii) the QLE is monotonic decreasing with respect to $r$. Note that at $a=0$ (corresponding to $M_{\mathrm{ir}}=m$, the spherical symmetric case), the outer horizon $r_{+}=2m$ is only one point, which is $\mathrm{QLE}=2m$.}\label{fig:4}
\end{figure}

\newpage

\section{Conclusion and discussion}
We analyze the critical value of the Kerr QLE (the Brown-York mass) under the choice for the surface with constant $t,r$ in the Boyer-Lindquist coordinates. It is known that the outer horizon of the Kerr black hole cannot be embedded isometrically into $\mathbb{R}^{3}$ if $a$ is too large. We found that $S(t,r_{k})$ is the limit of such an embedding: if $r<r_{k}(a)$ [see \eqref{eqn-rk}] then the Gauss curvature $K$ is negative at the poles, whereas $K>0$ implies the existence of an isometric embedding. We consider only the region outside the outer horizon, i.e.\ $r\geq r_{+}$, in which the embedding is guaranteed for $a<\sqrt{3}m/2$. If $a>\sqrt{3}m/2$, the isometric embedding exists only for $r>r_{k}(a)$. 

Regarding the positivity of the critical value for the Kerr QLE, it is obvious that if the integrand $k_{0}-k>0$ then the positivity is satisfied. This is guaranteed for small rotation, but not for large $a$. In fact, $k_{0}-k>0$ is valid on the \emph{whole} surface only in the region $r>\sqrt{3}a$. It is interesting that the three curves, $r_{+}(a)$, $r_{k}(a)$ and $r=\sqrt{3}a$ intersect at the point $a=\sqrt{3}m/2$. This implies that $k_{0}-k>0$ in the region $r>r_{h}(a)$ [see \eqref{eqn-rh}]. For non-slow rotation, in the region $r_{k}(a)<r<\sqrt{3}a$, the integrand $k_{0}-k$ is not positive everywhere on the surface, but the numerical results imply that after integration over the surface, the QLE is positive in the triangle-like region. Furthermore, it is monotonically decreasing with respect to $r$.

To answer any concern about the decreasing QLE corresponding to an increasing angular momentum $a$, one may be inspired from \emph{the second law} of black hole dynamics, which implies that the black hole area \emph{can never decrease}. The area of a Kerr black hole is $A=4\pi(r^{2}_{+}+a^2)=8\pi m r_{+}$ [\cite{MTW}, Box 33.4]. One may imagine that there are two black holes with the same mass $m$ but different angular momenta, say $a_{1}$ and $a_{2}$, it is not difficult to see that $a_{1}<a_{2}$ implies $A_{1}>A_{2}$. It is reasonable that a larger black hole would carry larger quasi-local energy. And also the  black hole area is proportional to the square of the irreducible mass. So, if we replace $a$ by the irreducible mass $M_{\mathrm{ir}}$, the QLE becomes increasing with respect to $M_{\mathrm{ir}}$.

By a results of Chen-Wang-Yau, the critical value of the Kerr QLE is a global minimum (with respect to the embedding) for the region $r>r_{h}(a)$, and an interesting question is whether this is also true in the triangle-like region for non-slow rotation. 

In the region $r<r_{k}(a)$, the isometric embedding into $\mathbb{R}^3$ does not exist, which means one could try to find a non-constant solution for $\tau$. This is another interesting problem.

\section*{Acknowledgement}
We would like to thank Professor James M. Nester and Professor Naqing Xie for helpful suggestions and discussions. We would like to thank Dr. Gang Sun for sharing his codes for the numerical integration program.

C. Yu is partially supported by the Yangfan project of Guangdong Province and NSFC 11571215. J. L. Liu is supported by the China Postdoctoral Science Foundation 2016M602497 and partially supported by NSFC 61601275.


\end{document}